\documentclass{article}
\usepackage[T2A]{fontenc}
\usepackage[cp1251]{inputenc}
\usepackage[english,russian]{babel}
\usepackage[tbtags]{amsmath}
\usepackage{amsfonts,amssymb,mathrsfs,amscd}
\usepackage[hyper]{msb-a}
\JournalName{}
\numberwithin{equation}{section}
\theoremstyle{plain}
\newtheorem{theorem}{Теорема}
\newtheorem{lemma}{Лемма}[section]

\theoremstyle{definition}
\newtheorem{definition}{Определение}
\newtheorem{proof}{Доказательство}
\newtheorem{remark}{Замечание}

\newcommand{\mb}{\phantom{\displaystyle 1_{\displaystyle 1_{\displaystyle 1_{\displaystyle 1}}}}}

\newcommand{\mr}{\phantom{z}}

\begin{document}

\title{Метрическое описание изгибаемых октаэдров}
\author[S.\,N.~Mikhalev]{С.\,Н.~Михалев}
\address{МГУ имени М.\,В.~Ломоносова\\ Российский университет дружбы народов}
\email{mikhalev@bk.ru}

\date{06.11.2019}
\udk{514.113.5}

\maketitle

\begin{fulltext}

\begin{abstract}
Найдено новое описание изгибаемых октаэдров Брикара с использований условий в терминах длин ребер, пригодное для исследования ряда задач метрической геометрии октаэдров, в частности, для поиска доказательства гипотезы И.\,Х.~Сабитова о равенстве нулю всех (кроме старшего) коэффициентов многочлена для объема октаэдра 3-го типа.

Библиография: 11 названий.
\end{abstract}


\markright{Метрическое описание изгибаемых октаэдров}

\section{История вопроса}
\label{sec0}
Одним из центральных объектов исследований в метрической теории многогранников являются изгибаемые многогранники. Первые примеры нетривиально изгибаемых многогранников в $\mathbb{R}^3$ были построены Брикаром в его работе~\cite{bric}, где дается классификация изгибаемых многогранников, комбинаторно эквивалентных правильному октаэдру. Брикар выделил три типа октаэдров, допускающих нетривиальное изгибание, из которых первые два имеют простое описание, а третий тип устроен довольно сложным образом. Простота изучаемого объекта (у октаэдра всего шесть вершин!) в сочетании с нетривиальностью результата (как в части получения, так и в части даже его описания) подталкивала в дальнейшем и других исследователей находить все новые и новые способы "повторения"\ результатов Брикара и их интерпретации: упомянем работы Беннета~\cite{ben}, Лебега~\cite{leb}, Штахеля~\cite{sta} и, совсем свежую работу группы европейских математиков~\cite{avstr}.

Развитие метрической геометрии многогранников получило новый мощный импульс с доказательством в середине 1990-х годов И.Х.~Сабитовым~\cite{sab1,sab2,sab3} гипотезы кузнечных мехов о постоянстве обобщенного объема изгибаемого многогранника. Использованный в этих работах метод (его можно называть методом "геометрии расстояний"), основанный на применении соотношений Кэли-Менгера, был задействован позднее в работах~\cite{sab4,mih2} и потенциально может быть использован при решении широкого круга задач метрической геометрии многогранников. Однако, при этом все метрические характеристики многогранников должны быть заданы в терминах {\it расстояний} (длин ребер и диагоналей).

Известные описания октаэдров Брикара даются в других терминах. Например, сам Брикар дает оперирует {\it углами} --- плоскими углами граней, а также углами, получаемыми в ходе некоторых дополнительных построений. Правда, описание первых двух типов он все же дает и на языке "геометрии расстояний", благо это делается несложно. Однако, в сложном случае октаэдра третьего типа попытка непосредственного перевода соответствующих уравнений на язык "геометрии растояний"\ приводит к комбинации уравнений, совместное использование которых представляется затруднительным (см., например, условие (13) на стр. 19 работы Брикара~\cite{bric}).

В работе Беннета при описании октаэдра третьего типа появляются новые (их не было у Брикара) необходимые условия на длины ребер: три условия нулевой алгебраической суммы ребер каждого из экваторов и четыре специальных условия "пропорциональности"\ ребер вида $a\cdot b\cdot c=p\cdot q\cdot r$. Однако Беннет не замечает, что независимыми из четырех условий пропорциональности являются только три (любое из этих четырех условий является следствием трех других), и, (ошибочно) считая, что число найденных {\it независимых} необходимых условий на двенадцать ребер октаэдра равно семи ("три условия нулевой суммы и четыре условия пропорциональности"), делает вывод о пятипараметричности множества октаэдров третьго типа --- впрочем, это утверждение, тем не менее, все равно оказывается верным, так как длины ребер октаэдра третьего типа подчиняются еще одному уравнению, которое найдено в нашей работе (см. утверждение 4 теоремы~\ref{t5}), и которое не является следствием условий нулевой алгебраической суммы длин ребер экваторов и условий пропорциональности (см. замечание~\ref{indp}).

Не останавливаясь на анализе остальных работ, скажем, что, в целом, с 1897 года так и не появилось работы, в которой бы давалось такое описание изгибаемых октаэдров, которое позволяло бы применять к изучению и описанию их свойств методы "геометрии расстояний". Наша настоящая работа как раз и призвана восполнить этот пробел, открывая, в частности, дорогу, к вычислению для изгибаемых октаэдров всех трех типов многочлена Сабитова для объема. В качестве первого шага в этом направлении мы также вычисляем два из восьми коэффициентов этого унитарного многочлена, имеющего в случае октаэдра степень 8 (оказалось, оба они равны нулю).

Для того, чтобы получить описание изгибаемых октаэдров Брикара на языке "геометрии расстояний", мы, не пользуясь результатами Брикара, находим новое описание всевозможных изгибаемых октаэдров, сразу в удобной метрической записи (теоремы \ref{t2}--\ref{t5}), попутно получив условия приводимости многочлена Кэли-Менгера, представляющие также и самостоятельный интерес (теорема~\ref{t1}).

\section{Обозначения и определения}
\label{sec0}

Рассмотрим метрический симплициальный комплекс $K$, комбинаторно эквивалентный правильному октаэдру, ребрам которого приписаны так как указано на рис.~\ref{oct} строго положительные удовлетворяющие строгим неравенствам треугольника на гранях числа (длины ребер). Под нетривиальным изгибанием понимается непрерывное семейство изометрических реализаций $P_t:K\to \mathbb{R}^3$, такое, что при изменении $t$ непрерывно меняются все три диагонали октаэдра (диагоналями называются отрезки между вершинами, которые не соединены ребрами). Наше определение нетривиального изгибания отличается от классического, когда требуется, чтобы менялась {\it хотя бы одна} диагональ. Мы тем самым заранее исключаем из рассмотрения известные случаи реализации комплекса $K$: (1) в виде дважды покрытого четырехгранного угла и (2) в виде пары смежных по ребру граней, на каждую из которых накладываются три другие грани.

Из известной теоремы Глюка вытекает, что в общем случае любая реализация комплекса $K$ получается неизгибаемой, а изгибаемость возможна лишь при выполнении определенных условий на метрику (длины ребер) комплекса $K$. Мы найдем условия на длины ребер, выполнение которых {\it необходимо} для существования изометрической реализации комплекса $K$ в виде нетривиально изгибаемого октаэдра, и исходя из этих условий выделим классы изгибаемых октаэдров.

В дальнейшем будем, как правило, использовать одинаковые обозначения для метрических комплексов и их реализаций в $\mathbb{R}^3$. Кроме того, примем соглашение, что одна и та же строчная буква обозначает и длину отрезка (число), и сам этот отрезок как геометрический объект. Далее, две или более подряд записанные строчные буквы будут обозначать соответствующий одномерный метрический комплекс (или его изометрический образ). Например, $pcq$ обозначает треугольник $v_1v_2v_3$, а $cfhg$ --- реберный цикл $v_2v_3v_4v_5$.

Обозначим {\it квадраты} длин диагоналей $v_2v_4$, $v_3v_5$, $v_1v_6$ буквами $x$, $y$, $z$, соответственно (снова теми же буквами будем обозначать также и сами диагонали как геометрические объекты). Будем мыслить ребра октаэдра как известные величины (совокупность их длин определяет метрику октаэдра и будет обозначаться $l$), а диагонали $x$, $y$, $z$ --- как неизвестные.

\begin{figure}
    \begin{picture}(200,180)
    \put(-10,0){
        \begin{picture}(200,180)
            {\thicklines
                \put(100,150){\line(1,-1){60}} \put(100,150){\line(1,-4){20}} \multiput(100,150)(-1,-3){20}{\circle*{1}} \put(100,150){\line(-3,-4){60}} \put(40,70){\line(1,0){80}} \multiput(40,70)(4,2){10}{\circle*{1}} \put(120,70){\line(2,1){40}} \multiput(80,90)(4,0){20}{\circle*{1}}
                \put(100,10){\line(-1,1){60}} \multiput(100,10)(-1,4){20}{\circle*{1}} \put(100,10){\line(1,3){20}} \put(100,10){\line(3,4){60}}
            }

            {\thinlines
                \put(96,153){${\scriptstyle v_1}$} \put(33,65){${\scriptstyle v_2}$} \put(121,65){${\scriptstyle v_3}$} \put(158,90){ ${\scriptstyle v_4}$} \put(81,92){ ${\scriptstyle v_5}$} \put(99,4){${\scriptstyle v_6}$}
                \put(100,150){\circle*{2}} \put(40,70){\circle*{2}} \put(120,70){\circle*{2}} \put(160,90){\circle*{2}} \put(80,90){\circle*{2}} \put(100,10){\circle*{2}}

                \put(64,113){$p$} \put(112,113){$q$} \put(131,121){$r$} \put(84,120){$s$}
                \put(75,63){$c$} \put(136,71){$f$} \put(131,91){$h$} \put(62,86){$g$}
                \put(64,33){$b$} \put(102,33){$a$} \put(131,43){$e$} \put(83,43){$d$}


                \put(30,140){$K$}
            }
        \end{picture}
    }
    \put(180,0){
        \begin{picture}(200,180)
            {\thicklines
                \put(100,150){\line(1,-1){60}} \put(100,150){\line(1,-4){20}} \multiput(100,150)(-1,-3){20}{\circle*{1}} \put(100,150){\line(-3,-4){60}} \put(40,70){\line(1,0){80}} \multiput(40,70)(4,2){10}{\circle*{1}} \put(120,70){\line(2,1){40}} \multiput(80,90)(4,0){20}{\circle*{1}}
            }

            {\thinlines

                \put(64,113){$p$} \put(112,113){$q$} \put(131,121){$r$} \put(84,120){$s$}
                \put(75,63){$c$} \put(136,71){$f$} \put(131,91){$h$} \put(62,86){$g$}

                \put(40,70){\line(6,1){120}} \put(120,70){\line(-2,1){40}}
                \put(73,77){${\scriptstyle \mathbf x}$} \put(94,85){${\scriptstyle \mathbf y}$}

                \put(30,140){$K_1$}
            }
        \end{picture}
    }
    \end{picture}
    \caption{}
    \label{oct}
\end{figure}

Цикл из четырех ребер, никакие два из которых не инцидентны одной и той же грани, будем называть {\it экватором} октаэдра. Метрический симплициальный комплекс $K_j\subset K$, являющийся звездой вершины $v_j$ комплекса $K$, будем называть {\it четырехгранным углом с вершиной $v_j$}. Диагональю произвольного комплекса мы называем отрезок, соединяющий вершины, не соединенные ребром. В соответствии с этими определениями у октаэдра три диагонали, три экватора (в каждом две диагонали) и шесть четырехгранных углов (в каждом две диагонали). Каждый четырехгранный угол "опирается"\ на некоторый экватор, являющийся для него краем.

Пусть длины ребер некоторого экватора равны $a$, $b$, $c$, $d$. Будем называть экватор {\it метрически симметричным}, если выполняется хотя бы одно из трех условий: $a=c$, $b=d$, или $a=b$, $c=d$, или $a=d$, $b=c$. Будем говорить, что экватор {\it имеет нулевую сумму}, если $(a+b-c-d)(a-b+c-d)(a-b-c+d)=0$. Таким образом, метрически симметричный экватор всегда имеет нулевую сумму, обратное, вообще говоря, неверно.

\section{Условия приводимости многочлена Кэли-Менгера}
\label{sec2}

Известно (см.~\cite{berze}), что попарные расстояния $d_{ij}=d_{ji}, i,j=1..5$ между пятью точками в $\mathbb{R}^3$ удовлетворяют следующему условию (это равенство часто называют уравнением Кэли-Менгера, а его левую часть -- определителем или многочленом Кэли-Менгера):
\begin{equation}
\label{km}
\left|
\begin{array}{cccccc}
0& 1& 1& 1& 1& 1\\
1& 0& d_{12}^2& d_{13}^2& d_{14}^2& d_{15}^2\\
1& d_{21}^2& 0& d_{23}^2& d_{24}^2& d_{25}^2\\
1& d_{31}^2& d_{32}^2& 0& d_{34}^2& d_{35}^2\\
1& d_{41}^2& d_{42}^2& d_{43}^2& 0& d_{45}^2\\
1& d_{51}^2& d_{52}^2& d_{53}^2& d_{54}^2& 0
\end{array}
\right|=0
\end{equation}

Напомним, выполнение условия~(\ref{km}) {\it необходимо} для того чтобы данные числа $d_{ij}$ являлись попарными расстояниями для некоторых пяти точек. Достаточным это условие само по себе не является: чтобы гарантировать существование соответствующей пятерки точек нужно наложить на числа $d_{ij}$ некоторые дополнительные условия.

Рассмотрим четырехгранный угол $K_1$ (см. рис.~\ref{oct}). Запишем умноженный на -1 многочлен Кэли-Менгера~(\ref{km}) для расстояний между его вершинами по степеням $x$ и $y$:
\begin{equation}
\label{5versh}
Q(x,y)=x^2y^2-2(s^2+q^2)x^2y-2(p^2+r^2)x y^2+(s^2-q^2)^2x^2+(p^2-r^2)^2y^2+...
\end{equation}
(Иногда, как здесь, мы будем опускать часть слагаемых и заменять их многоточием для удобства восприятия и экономии места.)

Отметим, что в отличие от замкнутых поверхностей, которые почти все являются неизгибаемыми, любой четырехгранный угол (за исключением некоторых очевидных вырожденных случаев) допускает нетривиальное изгибание. Условие Кэли-Менгера можно трактовать как уравнение, описывающее изгибание четырехгранного угла $K_1$, форма которого определяется диагоналями $x$ и $y$. Такая интерпретация наглядно объясняет тот факт, что многочлен Кэли-Менгера имеет вторую степень по каждой из переменных $x$ и $y$: для каждого значения одной из диагоналей существует (в невырожденных случаях) {\it два} возможных значения другой диагонали (пару граней $pqc$ и $psg$ можно отразить относительно плоскости $qsy$, а пару $pqc$ и $qrf$ --- относительно $prx$).

Нас будет интересовать существование нетривиального разложения многочлена Кэли-Менгера на рациональные относительно $x$ и $y$ множители, иными словами, приводимость многочлена в $\mathbb{R}[x,y]$. Такое разложение возможно в случае, когда длины ребер удовлетворяют некоторым соотношениям. Наша ближайшая цель -- найти эти соотношения и соответствующие им разложения многочлена Кэли-Менгера. Стоит отметить, что в общем случае определитель Кэли-Менгера абсолютно неприводим~\cite{andrea}.

Обозначим (см. рис.~\ref{oct}) $\alpha_1$, $\alpha_2$, $\alpha_3$, $\alpha_4$ величины плоских углов $pq$, $qr$, $rs$, $sp$, а $S_1$, $S_2$, $S_3$, $S_4$ --- площади треугольников $pqc$, $qrf$, $rsh$, $psg$. Из выполнения строгих неравенств треугольника на гранях следует, что $0<\alpha_k<\pi$, $S_k>0$, $k=1,2,3,4$.

Разобьем множество всех четырехгранных углов на попарно непересекающиеся классы I, IIx, IIy, III следующим образом:

\begin{definition}
\begin{equation}\label{d2x}
    \begin{array}{lllll}
        K_1\in IIx
        &
        \Leftrightarrow
        &
        \left\{\begin{array}{l}
            \alpha_1=\alpha_4 \\
            \alpha_2=\alpha_3 \\
            \alpha_1\ne \alpha_3
        \end{array}
        \right.
        &
        \mbox{или}
        &
        \left\{\begin{array}{l}
            \alpha_1=\pi-\alpha_4 \\
            \alpha_2=\pi-\alpha_3 \\
            \alpha_1\ne \pi-\alpha_3
        \end{array}
        \right.
    \end{array}
\end{equation}

\begin{equation}\label{d2y}
    \begin{array}{lllll}
        K_1\in IIy
        &
        \Leftrightarrow
        &
        \left\{\begin{array}{l}
            \alpha_1=\alpha_2 \\
            \alpha_3=\alpha_4 \\
            \alpha_1\ne \alpha_3
        \end{array}
        \right.
        &
        \mbox{или}
        &
        \left\{\begin{array}{l}
            \alpha_1=\pi-\alpha_2 \\
            \alpha_3=\pi-\alpha_4 \\
            \alpha_1\ne \pi-\alpha_3
        \end{array}
        \right.
    \end{array}
\end{equation}

\begin{equation}\label{d3}
    \begin{array}{lllll}
        K_1\in III
        &
        \Leftrightarrow
        &
        \left\{\begin{array}{l}
            \alpha_1=\alpha_3 \\
            \alpha_2=\alpha_4
        \end{array}
        \right.
        &
        \mbox{или}
        &
        \left\{\begin{array}{l}
            \alpha_1=\pi-\alpha_3 \\
            \alpha_2=\pi-\alpha_4
        \end{array}
        \right.
    \end{array}
\end{equation}
Во всех остальных случаях говорим, что $K_1\in I$.

Положим $II=IIx\cup IIy$. Если $K_1\in IIx$, будем также говорить, что $K_1$ принадлежит классу $II$ относительно диагонали $x$ (и аналогично для $K_1\in IIy$).
\end{definition}
\begin{remark}\label{equiv}
Для каждого условия на углы из определения 1 есть эквивалентная запись в терминах длин ребер. Например,
\begin{equation*}\label{a14}
    \begin{array}{lll}
        \alpha_1=\alpha_4
        &
        \Leftrightarrow
        &
        s(c^2-q^2-p^2)=q(g^2-p^2-s^2) \\
        \alpha_1=\pi-\alpha_4
        &
        \Leftrightarrow
        &
        s(c^2-q^2-p^2)=-q(g^2-p^2-s^2).

    \end{array}
\end{equation*}
\end{remark}

\begin{lemma}\label{lemmma1}
Если $K_1\in IIx$, то $psS_1=pqS_4$, и $rsS_2=rqS_3$. Если $K_1\in IIy$, то $qrS_1=pqS_2$, и $psS_3=srS_4$. Если $K_1\in III$, то $srS_1=pqS_3$, и $psS_2=rqS_4$.
\end{lemma}
(Для доказательства достаточно заметить, что, например, из $K_1\in IIx$ вытекает, что $\sin\alpha_1=\sin\alpha_4$, и $\sin\alpha_2=\sin\alpha_3$.)

\begin{theorem} \label{t1}
Пусть четырехгранный угол $K_1$ нетривиально изгибаем. Тогда справедливо в точности одно из утверждений:

1) $K_1\in I$, и многочлен $Q(x,y)$ неприводим.

2) $K_1\in IIx$, многочлен $Q(x,y)$ приводим, и если
\begin{equation}\label{cf2x}
    c^2=q^2+p^2\pm\frac{q}{s}(g^2-p^2-s^2), \mr f^2=q^2+r^2\pm\frac{q}{s}(h^2-r^2-s^2),
\end{equation}
(здесь и далее знаки в формулах выбираются одинаковые), то разложение $Q(x,y)$ на неприводимые множители имеет (с соответствующим выбором знаков) вид:
\begin{equation}\label{q2x}
    Q(x,y)=\left[x^2y-(s\pm q)^2x^2-2(p^2+r^2)xy+(p^2-r^2)^2y+...\right]\cdot\left[y-(q\mp s)^2\right].
\end{equation}

3) $K_1\in IIy$, многочлен $Q(x,y)$ приводим, и если
\begin{equation}\label{cf2y}
    f^2=r^2+q^2\pm\frac{r}{p}(c^2-q^2-p^2), \mr h^2=r^2+s^2\pm\frac{r}{p}(g^2-p^2-s^2),
\end{equation}
то разложение $Q(x,y)$ на неприводимые множители имеет вид:
\begin{equation}\label{q2y}
    Q(x,y)=\left[xy^2-(p\pm r)^2y^2-2(s^2+q^2)xy+(p^2-r^2)^2x+...\right]\cdot\left[x-(p\mp r)^2\right].
\end{equation}

4) $K_1\in III$, многочлен $Q(x,y)$ приводим, и если
\begin{equation}\label{cf3}
    c^2=q^2+p^2\pm\frac{pq}{sr}(h^2-r^2-s^2), \mr f^2=q^2+r^2\pm \frac{qr}{ps}(g^2-p^2-s^2), \mbox{то}
\end{equation}
\begin{equation}\label{q3}
    Q(x,y)=\left[xy-(s\pm q)^2x-(p+r)^2y+...\right]\cdot\left[xy-(s\mp q)^2x-(p-r)^2y+...\right].
\end{equation}
\end{theorem}
\begin{remark}
Формулировка теоремы корректна, так как из четырех утверждений никакие два не могут быть справедливыми одновременно.
\end{remark}
\begin{remark}
В~(\ref{q3}) неприводимость сомножителей не гарантируется.
\end{remark}
\begin{proof}
Простая проверка показывает, что подстановка~(\ref{cf2x}) в~(\ref{5versh}) дает~(\ref{q2x}). Аналогично, подстановка~(\ref{cf2y}) в~(\ref{5versh}) дает~(\ref{q2y}), а подстановка~(\ref{cf3}) в~(\ref{5versh}) дает~(\ref{q3}).

Далее, дискриминант $\Delta(x)$ квадратичного многочлена~(\ref{5versh}) относительно переменной $y$ имеет вид
\begin{equation}\label{deltax}
    \Delta(x) = 16 P_1(x)P_2(x),
\end{equation}
где
$$
\begin{array}{l}
P_1(x)=s^2x^2+(-g^2r^2+s^4-s^2h^2-p^2s^2+p^2r^2-p^2h^2+g^2h^2-g^2s^2-s^2r^2)x+\\
+p^4h^2-p^2h^2g^2+r^2h^2s^2-p^2h^2s^2+r^4g^2-p^2r^2h^2+p^2g^2s^2+p^2h^4-\\
-p^2g^2r^2+r^2g^4-r^2g^2h^2-r^2g^2s^2, \mb \\
P_2(x)=q^2x^2+(-r^2q^2-q^2f^2-r^2c^2-p^2f^2+c^2f^2+q^4-c^2q^2-q^2p^2+p^2r^2)x+\\
+r^4c^2-r^2c^2p^2-f^2r^2c^2-p^2f^2c^2-p^2f^2r^2-c^2q^2r^2-p^2f^2q^2+r^2c^4+q^2f^2r^2+\\
+p^4f^2+p^2f^4+q^2c^2p^2.
\end{array}
$$

Дискриминант многочлена $P_1(x)$ равен $S_3^2 S_4^2>0$, а дискриминант многочлена $P_2(x)$ равен $S_1^2 S_2^2>0$. Старшие коэффициенты $P_1(x)$ и $P_2(x)$ положительны. Поэтому $\Delta(x)$ является полным квадратом тогда и только тогда, когда коэффициенты $P_1(x)$ и $P_2(x)$ пропорциональны, что эквивалентно тождественному (при всех $x$) выполнению равенства $q^2P_1(x)-s^2P_2(x)=0$ или тому, что \begin{equation}
\label{linsys}
\left\{
\begin{array}{l}
(s^2q^2-f^2s^2+s^2r^2)c^2-q^2p^2h^2-p^2r^2s^2-q^2g^2s^2+q^2p^2r^2+f^2q^2s^2-\\
-q^2h^2s^2-q^2g^2r^2+p^2f^2s^2+q^2h^2g^2+q^2s^4-s^2q^4=0 \mb\\
r^2s^2c^4+(-p^2f^2s^2-p^2r^2s^2-s^2q^2r^2-f^2s^2r^2+s^2q^2p^2+r^4s^2)c^2+\\
+r^2q^2p^2h^2+r^2q^2g^2s^2-q^2p^2g^2s^2+r^2q^2g^2p^2+g^2q^2p^2h^2+\\
+s^2q^2p^2h^2-r^2q^2g^4-p^2f^2s^2r^2+q^2h^2g^2r^2-r^4q^2g^2-q^2p^4h^2+\\
+p^4f^2s^2-q^2h^2s^2r^2+s^2r^2q^2f^2-s^2p^2f^2q^2-q^2p^2h^4+p^2f^4s^2=0
\end{array}
\right.
\end{equation}

Решим систему~(\ref{linsys}) относительно $c^2$ и $f^2$. Легко проверить, что система имеет в точности четыре решения: это пары $(c^2,f^2)$ из~(\ref{cf2x}) и~(\ref{cf3}).

Рассуждая аналогично, получим, что $\Delta(y)$ является полным квадратом тогда и только тогда, когда выполняются условия одной из четырех пар~(\ref{cf2y}) и~(\ref{cf3}).

Пусть дано нетривиальное разложение $Q(x,y)= Q_1(x,y)\cdot Q_2(x,y)$. Так как старший коэффициент в $Q(x,y)$ не равен нулю, обязательно реализуется один из случаев: \\
(А) один из множителей --- квадратный трехчлен относительно $x$, а другой --- квадратный трехчлен относительно $y$; \\
(Б) степени $Q_1$ и $Q_2$ относительно $y$ равны 1 $\Leftrightarrow$ корни $Q(x,y)$ как квадратного трехчлена относительно $y$ рационально выражаются через $x$ $\Leftrightarrow$ дискриминант $\Delta(x)$ является полным квадратом; \\
(В) степени $Q_1$ и $Q_2$ относительно $x$ равны 1 $\Leftrightarrow$ корни $Q(x,y)$ как квадратного трехчлена относительно $x$ рационально выражаются через $y$ $\Leftrightarrow$ дискриминант $\Delta(y)$ является полным квадратом.

Если не выполнено ни условие (Б), ни условие (В), тогда выполнено условие (А). Но тогда длины ребер таковы, что нетривиальное изгибание невозможно (каждая диагональ может принимать не более двух фиксированных значений), что противоречит условию теоремы.

Пусть выполнено условие (Б), но не выполнено условие (В). Это означает, что выполнено одно из условий~(\ref{cf2x}), и не выполнено ни одно из условий~(\ref{cf2y}) и~(\ref{cf3}). Отсюда следует, что $K_1\in IIx$, и что множитель третьей степени в~(\ref{q2x}) неприводим (приводимость означала бы, что корни $x$ рационально выражаются через $y$).

Аналогично, пусть выполнено условие (В), но не выполнено условие (Б). Это означает, что выполнено одно из условий~(\ref{cf2y}), и не выполнено ни одно из условий~(\ref{cf2x}) и~(\ref{cf3}). Отсюда следует, что $K_1\in IIy$, и что множитель третьей степени в~(\ref{q2y}) неприводим.

Наконец, пусть выполнено и условие (Б), и условие (В). Такое может произойти в двух случаях: (1) выполнено одно из условий~(\ref{cf3}), тогда $K_1\in III$, и выполняется~(\ref{q3}); (2) выполнено одно из условий~(\ref{cf2x}) и одно из условий~(\ref{cf2y}), откуда вытекает, что выполняется и~(\ref{cf3}), а значит снова верно~(\ref{q3}), и $K_1\in III$.

Итак, мы доказали, что приводимость $Q(x,y)$ эквивалентна выполнению одного из утверждений 2, 3, 4. Делаем отсюда вывод, что класс I состоит из тех и только тех четырехгранных углов, для которых $Q(x,y)$ неприводим, что и завершает доказательство теоремы.

\end{proof}

\begin{remark}
В качестве неизвестных, относительно которых решалась система~(\ref{linsys}), были выбраны переменные $c$ и $f$. Если выбрать переменные $g$ и $h$ (соответствующие ребрам, лежащим по другую сторону от диагонали $x$), приходим к тем же самым линейным соотношениям~(\ref{cf2x}) между $c^2$ и $g^2$ и между $f^2$ и $h^2$, отличие лишь в том, что разрешены они будут относительно $g^2$ и $h^2$.
\end{remark}

\begin{remark}
Каждому из двух множителей в разложении $Q(x,y)$ отвечает свое непрерывное семейство пространственных положений четырехгранного угла $K_1$. Например, это может быть выпуклый четырехгранный угол (изгибание описывается одним из множителей), и самопересекающийся четырехгранный угол, полученный из исходного путем отражения двух граней относительно "диагональной"\ плоскости (изгибание описывается вторым множителем). Линейному множителю при этом соответствует тривиальное изгибание. Так, в~(\ref{q2x}) один из множителей имеет вид $y-y_0$. Длины ребер здесь таковы, что отраженные грани накладываются на две другие грани, а одна из вершин экватора попадает на противолежащее ей ребро (это становится возможным в силу равенства плоских углов прилежащих граней). Такая фигура с самоналожениями изгибается тривиальным образом: в ходе изгибания диагональ $y$ не меняется. Поэтому нас будет интересовать только положение, соответствующее множителю третьей степени. В таких случаях для единообразия формулировок мы иногда будем вместо общего многочлена Кэли-Менгера сразу рассматривать его сомножитель третьей степени.
\end{remark}

\begin{remark} \label{privodimost}
Во всех случаях из неприводимости многочлена в $\mathbb{R}[x,y]$ следует его неприводимость в $\mathbb{C}[x,y]$. Это вытекает из представления дискриминанта $\Delta(x)$ в виде~(\ref{deltax}) и аналогичного представления для $\Delta(y)$: если бы, например, $Q(x,y)$ был неприводим над полем $\mathbb{R}$, но приводим над $\mathbb{C}$, это означало бы, что или $\Delta(x)=-P^2(x)$, где $P(x)\in\mathbb{R}[x]$, что невозможно, так как коэффициент при $x$ в старшей степени в $\Delta(x)$ равен $q^2s^2>0$, или $\Delta(y)=-\widetilde{P}^2(y)$, что, аналогично, невозможно.
\end{remark}

\section{Условия изгибаемости октаэдра}
\label{sec3}
В дальнейшем предполагаем, что нам дан комплекс $K$ (рис.~\ref{oct}), изометрически реализуемый в $\mathbb{R}^3$ в виде изгибаемого октаэдра, который мы также будем обозначать $K$.

Рассмотрим один из экваторов октаэдра $K$, например, $cfhg$. Он разбивает $K$ на два четырехгранных угла: $K_1$ и $K_6$ (такие четырехгранные углы, пятерки точек и их многочлены Кэли-Менгера будем называть {\it соответствующими} данному экватору). Каждое из соответствующих экватору $cfhg$ условий Кэли-Менгера задает некоторое алгебраическое многообразие в плоскости переменных $x,y$. Из нетривиальной изгибаемости октаэдра вытекает, что эти два многообразия должны пересекаться по нетривиальному (не сводящемуся к дискретному набору точек) алгебраическому многообразию.

Так как для любых $f,g\in\mathbb{C}[x,y]$, таких что $g$ не делится на $f$, из неприводимости $f$ следует конечность множества общих нулей $f$ и $g$ (это простой алгебраический факт), приходим к выводу, что либо оба уравнения Кэли-Менгера неприводимы (и тогда их коэффициенты должны быть попарно равны), либо оба они приводимы (и тогда, аналогично, должны быть равны степени и пропорциональны коэффициенты некоторой пары неприводимых сомножителей). Отметим также, что указанные сомножители должны иметь степень 2 или 3 по совокупности переменных $x$ и $y$ (линейному множителю соответствует лишь тривиальное изгибание). Аналогичные рассуждения справедливы и для остальных двух экваторов.

Получаем, что для каждого экватора соответствующие ему четырехгранные углы принадлежат одному и тому же классу (I, IIx, IIy или III).

Классом экватора будем называть класс соответствующих ему четырехгранных углов. Каждый экватор относится к одному из классов (I, IIx, IIy или III) независимо от других. Классификацию изгибаемых октаэдров будем строить исходя из полного перебора возможных случаев, который дается набором теорем 2--5.

\begin{theorem} \label{t2}
Если хотя бы два экватора принадлежат классу I, то все три экватора метрически симметричны.
\end{theorem}

\begin{proof}
Без ограничения общности считаем, что $cfhg,bpre\in I$.

Рассмотрим экватор $cfhg$. Соответствующие ему многочлены Кэли-Менгера (для четырехгранных углов $K_1$ и $K_6$) имеют вид:
$$Q_1(x,y)=x^2y^2-2(s^2+q^2)x^2y-2(p^2+r^2)x y^2+(s^2-q^2)^2x^2+(p^2-r^2)^2y^2+...,$$
$$Q_6(x,y)=x^2y^2-2(a^2+d^2)x^2y-2(b^2+e^2)x y^2+(a^2-d^2)^2x^2+(b^2-e^2)^2y^2+...,$$
По предположению многочлены $Q_1$ и $Q_6$ неприводимы, следовательно, должны иметь пропорциональные коэффициенты. Получаем четыре условия, которые удобно записать в виде двух систем уравнений:

\begin{equation*}
    \begin{array}{lll}
        \left\{
        \begin{array}{l}
            s^2+q^2=a^2+d^2\\
            (s^2-q^2)^2=(a^2-d^2)^2
        \end{array}
        \right.
        &
        \left\{
        \begin{array}{l}
            p^2+r^2=b^2+e^2\\
            (p^2-r^2)^2=(b^2-e^2)^2.
        \end{array}
        \right.
    \end{array}
\end{equation*}

Исследуя эти системы, убеждаемся что каждый из экваторов $sqad$ и $bpre$ метрически симметричен.

Повторяя это рассуждение для экватора $bpre$, получаем, что экватор $cfhg$ также метрически симметричен. Теорема доказана.
\end{proof}

\begin{theorem} \label{t3}
Если хотя бы два экватора принадлежат классу $II$, то все три экватора метрически симметричны.
\end{theorem}

\begin{proof}
Без ограничения общности считаем, что $cfhg,bpre\in II$.

Необходимо рассмотреть два случая (в остальных случаях рассуждения и формулы идентичны с точностью до переобозначения переменных).

{\bf Случай (А).} $cfhg,bpre\in IIx$. Запишем многочлены Кэли-Менгера для четырехгранных углов $K_1$, $K_6$, $K_3$, $K_5$:
$$Q_1(x,y)=x^2y-(s\pm q)^2x^2-2(p^2+r^2)xy+(p^2-r^2)^2y+...$$
$$Q_6(x,y)=x^2y-(a\pm d)^2x^2-2(b^2+e^2)xy+(b^2-e^2)^2y+...$$
$$Q_3(x,z)=x^2z-(s\pm d)^2x^2-2(h^2+g^2)xz+(h^2-g^2)^2z+...$$
$$Q_5(x,z)=x^2z-(a\pm q)^2x^2-2(c^2+f^2)xz+(c^2-f^2)^2z+...$$
(Знаки выбираются независимо, всего получается 16 вариантов.)

Приравниваем коэффициенты при $x^2$, $xy$ и $y$ первой пары уравнений и при $x^2$, $xz$, $z$ второй пары уравнений --- получаем шесть условий, которые удобно записать в виде трех систем уравнений

\begin{equation}\label{sys3}
    \begin{array}{lll}
        \left\{
        \begin{array}{l}
            (s\pm q)^2 = (a\pm d)^2\\
            (s\pm d)^2=(a\pm q)^2
        \end{array}
        \right.
        &
        \left\{
        \begin{array}{l}
            p^2+r^2=b^2+e^2\\
            (p^2-r^2)^2=(b^2-e^2)^2
        \end{array}
        \right.
        &
        \left\{
        \begin{array}{l}
            c^2+f^2=h^2+g^2\\
            (c^2-f^2)^2=(h^2-g^2)^2
        \end{array}
        \right.
    \end{array}
\end{equation}

Исследуя системы для $bpre$ и $cfhg$ в~(\ref{sys3}), установим, что экваторы $bpre$ и $cfhg$ метрически симметричны. Осталось проверить, что экватор $sqad$ также метрически симметричен. Отметим, что без дополнительных условий из~(\ref{sys3}) следует, вообще говоря, лишь, что экватор $sqad$ имеет нулевую сумму, но не его метрическая симметричность.

Большинство вариантов расстановки знаков в~(\ref{sys3}) приводит к вырожденныму случаям (например, из $s+q+a+d=0$ следует равенство ребер нулю, а $s-q-a-d=0$ соответствует тривиальному изгибанию). Допустимыми являются лишь четыре варианта расстановки знаков. Для каждого из этих вариантов в дополнение к условиям~(\ref{sys3}) найдем одно дополнительное условие на длины ребер $sqad$. Вспомним, что каждому из четырех многочленов Кэли-Менгера соответствует пара условий на длины ребер четырехгранного угла, таким образом, должны выполняться одновременно восемь условий. Для каждой такой системы из восьми уравнений укажем способ получения уравнения-следствия в такой форме, чтобы его было удобно скомбинировать с системой для $sqad$ в~(\ref{sys3}).

Приведем выкладки только для одного из четырех случаев. Пусть, например,
$$Q_1(x,y)=x^2y-(s-q)^2x^2-2(p^2+r^2)xy+(p^2-r^2)^2y+...$$
$$Q_6(x,y)=x^2y-(a-d)^2x^2-2(b^2+e^2)xy+(b^2-e^2)^2y+...$$
$$Q_3(x,z)=x^2z-(s+d)^2x^2-2(h^2+g^2)xz+(h^2-g^2)^2z+...$$
$$Q_5(x,z)=x^2z-(a+q)^2x^2-2(c^2+f^2)xz+(c^2-f^2)^2z+...$$

Это соответствует следующим условиям на ребра:
$$c^2=q^2+p^2-\frac{q}{s}(g^2-p^2-s^2), \mr f^2=q^2+r^2-\frac{q}{s}(h^2-r^2-s^2),$$
$$c^2=a^2+b^2-\frac{a}{d}(g^2-b^2-d^2), \mr f^2=a^2+e^2-\frac{a}{d}(h^2-e^2-d^2),$$
$$b^2=a^2+c^2+\frac{a}{q}(p^2-q^2-c^2), \mr e^2=a^2+f^2+\frac{a}{q}(r^2-q^2-f^2),$$
$$b^2=d^2+g^2+\frac{d}{s}(p^2-s^2-g^2), \mr e^2=d^2+h^2+\frac{d}{s}(r^2-s^2-h^2).$$

Сложим уравнения левого столбца (формулы для $c^2$ и $b^2$), предварительно умножив их, соответственно, на $-asd$, $qsd$, $qsd$, $asq$. После преобразований получим:
\begin{equation}\label{sqad1}
    2sqad(s+q-a-d)=0 \Rightarrow s+q-a-d=0.
\end{equation}

Легко проверить что при выполнении условия~(\ref{sqad1}) системе для $sqad$ в~(\ref{sys3}) удовлетворяют только метрически симметричные экваторы. Случай (А) завершен.

{\bf Случай (Б).} $cfhg\in IIx$, и $bpre\in IIz$. Рассуждение в этом случае сходно с рассуждением в случае (А).

Запишем многочлены Кэли-Менгера для четырехгранных углов $K_1$, $K_6$, $K_3$, $K_5$:
$$Q_1(x,y)=x^2y-(s\pm q)^2x^2-2(p^2+r^2)xy+(p^2-r^2)^2y+...$$
$$Q_6(x,y)=x^2y-(a\pm d)^2x^2-2(b^2+e^2)xy+(b^2-e^2)^2y+...$$
$$Q_3(x,z)=z^2x-(h\pm g)^2z^2-2(s^2+d^2)zx+(s^2-d^2)^2x+...$$
$$Q_5(x,z)=z^2x-(c\pm f)^2z^2-2(a^2+q^2)zx+(a^2-q^2)^2x+...$$

В каждом из 16 вариантов расстановки знаков приравниваем коэффициенты при $x^2$, $xy$ и $y$ первой пары уравнений и при $z^2$, $zx$, $x$ второй пары уравнений --- получаем шесть условий, которые удобно записать в виде трех систем уравнений:

\begin{equation}\label{sys4}
    \begin{array}{lll}
        \left\{
        \begin{array}{l}
            (s\pm q)^2 = (a\pm d)^2\\
            (h\pm g)^2=(c\pm f)^2
        \end{array}
        \right.
        &
        \left\{
        \begin{array}{l}
            p^2+r^2=b^2+e^2\\
            (p^2-r^2)^2=(b^2-e^2)^2,
        \end{array}
        \right.
        &
        \left\{
        \begin{array}{l}
            s^2+d^2=a^2+q^2\\
            (s^2-d^2)^2=(a^2-q^2)^2.
        \end{array}
        \right.
    \end{array}
\end{equation}

Исследуя системы для $bpre$ и $sqad$ в~(\ref{sys4}), установим, что экваторы $bpre$ и $sqad$ метрически симметричны. Осталось проверить, что экватор $cfhg$ также метрически симметричен.

В большинстве случаев сочетание знаков таково, что решение соответствует лишь вырожденному или тривиально изгибающемуся октаэдру. Для всех оставшихся допустимых случаев (их четыре) получим теперь еще одно общее условие на длины ребер $cfhg$. Метод получения дополнительного условия здесь другой, отличный от примененного в случае (А).

В рассматриваемом случае $K_1, K_6\in IIx$, $K_3, K_5\in IIz$. По лемме 1:
$$\left\{\begin{array}{l}psS_1-pqS_4=0\\ shS_4-gsS_3=0\\ qrS_3-rsS_2=0\\ qcS_2-qfS_1=0\end{array}\right.$$

Перемножая эти уравнения, получаем, что
\begin{equation}\label{cfhg1}
    prs^2q^2(hc-fg)S_1S_2S_3S_4=0 \Leftrightarrow hc=fg.
\end{equation}

Легко проверить что при выполнении условия~(\ref{cfhg1}) уравнению для $cfhg$ первой из систем~(\ref{sys4}) удовлетворяют только метрически симметричные экваторы. Случай (Б) завершен. Теорема доказана.

\begin{theorem} \label{t4}
Если у октаэдра есть по одному экватору каждого из классов I, II и III, то все три экватора метрически симметричны.
\end{theorem}

Без ограничения общности считаем, что $sqad\in I$, $bpre\in II$, $cfhg\in III$. Повторяя рассуждение из доказательства теоремы~\ref{t2} для экватора $sqad$, выписываем системы

\begin{equation}\label{sys5}
    \begin{array}{ll}
        \left\{
        \begin{array}{l}
            c^2+g^2=f^2+h^2\\
            (c^2-g^2)^2=(f^2-h^2)^2,
        \end{array}
        \right.
        &
        \left\{
        \begin{array}{l}
            b^2+p^2=r^2+e^2\\
            (b^2-p^2)^2=(r^2-e^2)^2
        \end{array}
        \right.
    \end{array}
\end{equation}
и получаем, что экваторы $cfhg$ и $bpre$ метрически симметричны. Осталось доказать, что экватор $sqad$ также метрически симметричен.

Необходимо рассмотреть два случая.

{\bf Случай (А).} $bpre\in IIz$. Подобно тому, как это делалось при доказательстве теоремы~\ref{t2}, получаем отсюда, что экватор $sqad$ метрически симметричен.

{\bf Случай (Б).} $bpre\in IIx$. В этом случае $K_3, K_5\in IIx$, $K_1, K_6\in III$. По лемме 1
$$\left\{\begin{array}{l}acS_1-qcS_5=0\\ edS_5-baS_7=0\\ shS_7-hdS_3=0\\ pqS_3-rsS_1=0\end{array}\right.$$
Перемножая эти уравнения, получаем, что
\begin{equation}\label{epbr}
    acdshq(ep-br)S_1S_3S_5S_7=0 \Rightarrow ep=br.
\end{equation}

Ранее мы уже установили, что из~(\ref{sys5}) следует метрическая симметричность экватора $bpre$. Очевидно, что на самом деле справедливо более точное утверждение: из~(\ref{sys5}) следует, что или $b=r, p=e$, или $b=e, p=r$. С учетом условия~(\ref{epbr}) отсюда следует, что в рассматриваемом случае обязательно
\begin{equation}\label{be}
    b=e.
\end{equation}

Запишем теперь многочлены Кэли-Менгера для четырехгранных углов $K_1$, $K_6$, $K_3$, $K_5$:
$$Q_1(x,y)=\left[xy-(s\pm q)^2x-(p+r)^2y+...\right]\cdot\left[xy-(s\mp q)^2x-(p-r)^2y+...\right]$$
$$Q_6(x,y)=\left[xy-(a\pm d)^2x-(b+e)^2y+...\right]\cdot\left[xy-(a\mp d)^2x-(b-e)^2y+...\right]$$
$$Q_3(x,z)=x^2z-(s\pm d)^2x^2-2(h^2+g^2)xz+(h^2-g^2)^2z+...$$
$$Q_5(x,z)=x^2z-(a\pm q)^2x^2-2(c^2+f^2)xz+(c^2-f^2)^2z+...$$

В качестве уравнения изгибания угла $K_1$ ($K_6$) можно взять любой из двух множителей в $Q_1$ ($Q_6$) с любым выбором знака (конечно, если множитель неприводим; если приводимы оба --- изгибание тривиально). С учетом знаков в $Q_3$ и $Q_5$ всего максимально получается 64 варианта, из которых 8 являются допустимыми (они не приводят к заведомо вырожденным случаям). Каждому варианту соответствует набор условий на ребра четырехгранных углов. В каждом из случаев приравнивая коэффициенты при $x$ выбранных множителей в $Q_1$ и $Q_6$, а также при $x^2$ в $Q_3$ и $Q_5$, получим систему уравнений для ребер экватора $sqad$:
\begin{equation}
\begin{array}{l}
\left\{\begin{array}{l}
(s\pm q)^2 = (a\pm d)^2\\
(s\pm d)^2=(a\pm q)^2
\end{array}\right.
\end{array}
\end{equation}
и найдем дополнительное условие на длины ребер $sqad$. Мы ограничимся рассмотрением одного из случаев, в остальных рассуждение аналогичное.

Пусть, например, условия на ребра четырехгранных углов $K_1$, $K_6$, $K_3$, $K_5$ таковы:
\begin{equation}\label{cf4b}
    \begin{array}{ccc}
      c^2=a^2+b^2-\frac{ab}{de}(h^2-d^2-e^2) & \mr & f^2=a^2+e^2-\frac{ae}{bd}(g^2-b^2-d^2) \\
      c^2=q^2+p^2-\frac{pq}{sr}(h^2-r^2-s^2) & \mr & f^2=q^2+r^2-\frac{qr}{ps}(g^2-p^2-s^2) \\
      b^2=a^2+c^2+\frac{a}{q}(p^2-q^2-c^2) & \mr & e^2=a^2+f^2+\frac{a}{q}(r^2-q^2-f^2) \\
      b^2=d^2+g^2+\frac{d}{s}(p^2-s^2-g^2) & \mr & e^2=d^2+h^2+\frac{d}{s}(r^2-s^2-h^2),
    \end{array}
\end{equation}
что соответствет системе
\begin{equation}\label{sqadsys4b}
\begin{array}{lll}
\left\{\begin{array}{l}
(s-q)^2=(a-d)^2 \\
(s+d)^2=(a+q)^2
\end{array}
\right.
&
\Leftrightarrow
&
\left\{\begin{array}{l}
(s-q-a+d)(s-q+a-d)=0\\
(s-q-a+d)(s+q+a+d)=0
\end{array}
\right.
\end{array}
\end{equation}

Рассмотрим многочлены
\begin{equation}\label{alpha}
    \begin{array}{l}
      \alpha_1(l)=dqes(p-r)(s-d)(q-a) \\
      \alpha_2(l)=dqes(q^2r-ps^2-d^2p+2pds+ra^2-2arq) \\
      \alpha_3(l)=seaq(q^2r-rd^2+pds+sdr-ps^2+ra^2-2arq) \\
      \alpha_4(l)=s^2eaq(p-r)(q-a) \\
      \alpha_5(l)=dq^2se(p-r)(s-d) \\
      \alpha_6(l)=qdes(2pds-aqp+q^2r-ps^2-d^2p-arq+a^2p) \\
      \alpha_7(l)=seaq(q^2r-rd^2+pds+sdr-ps^2+ra^2-2arq) \\
      \alpha_8(l)=0
    \end{array}
\end{equation}

Простая проверка показывает, что числитель линейной комбинации уравнений~(\ref{cf4b}) с коэффициентами~(\ref{alpha}) после подстановки~(\ref{be}) и преобразований имеет вид:
\begin{equation}\label{sqad4b}
dqsea(s-q+a-d)(s+q-a-d)^2(r+p)=0 \Leftrightarrow \left[\begin{array}{l} s-q+a-d=0\\ s+q-a-d=0 \end{array}\right.
\end{equation}

Любое из двух уравнений в~(\ref{sqad4b}) в сочетании с системой~(\ref{sqadsys4b}) приводит нас к тому, что экватор $sqad$ метрически симметричен. Случай (Б) завершен. Теорема доказана.
\end{proof}

Введем обозначение: $B_{\varepsilon_1, \varepsilon_2, \varepsilon_3}=abc(h^2+s^2+r^2)-hsr(a^2+b^2+c^2)+
\varepsilon_1\cdot(csb(h^2+r^2-s^2)-hra(b^2+c^2-a^2))+
\varepsilon_2\cdot(ahb(s^2+r^2-h^2)-src(a^2+b^2-c^2))+
\varepsilon_3\cdot(cra(s^2+h^2-r^2)-shb(a^2+c^2-b^2)).$

\begin{theorem} \label{t5}
Если хотя бы два экватора принадлежат классу $III$, то либо все экваторы метрически симметричны, либо справедливы утверждения:

1) третий экватор также принадлежит классу III,

2) $qge=pfd, \mr bhq=cdr, \mr rga=sfb, \mr ahp=ces,$

3) все экваторы имеют нулевую сумму,

4) $B_{1,1,1}\cdot B_{1,-1,-1}\cdot B_{-1,1,-1}\cdot B_{-1,-1,1}=0,$

\end{theorem}
\begin{proof}
Без ограничения общности считаем, что $cfhg, bpre\in III$. Тогда уравнения Кэли-Менгера для $K_1$ и $K_3$ имеют вид:
$$Q_1(x,y)=\left[xy-(s\pm q)^2x-(p+r)^2y+...\right]\cdot\left[xy-(s\mp q)^2x-(p-r)^2y+...\right]$$
$$Q_3(x,z)=\left[xz-(a\pm q)^2x-(c+f)^2z+...\right]\cdot\left[xz-(a\mp q)^2x-(c-f)^2z+...\right]$$

Каждое условие разлагается на два одинаковых по структуре множителя. Изгибание четырехгранного угла в некотором конкретном положении описывается одним из двух (неприводимых) множителей. Таким образом, $y$ дробно-линейно выражается через $x$, а $x$ дробно-линейно выражается через $z$. Тогда $y$ дробно-линейно выражается через $z$, что по теореме~\ref{t1} влечет $sqad\in III$. Утверждение 1 доказано.

По лемме 1
$$\left\{\begin{array}{l}afS_1=cqS_6\\ hrS_6=feS_3 \\  pqS_3=srS_1\end{array}\right.$$
Перемножая эти уравнения, получаем, что
\begin{equation*}
    qrf(ahp-ces)S_1S_3S_6=0 \Rightarrow ahp=ces.
\end{equation*}
Аналогично получаются и остальные равенства утверждения 2. Отметим, что любое из четырех равенств утверждения 2 является следствием трех других.

Выпишем три пары условий Кэли-Менгера (для каждого из экваторов). В каждом условии можно, независимо, выбрать один из (неприводимых) множителей и знак в нем. Количество возможных систем (каждая состоит из трех пар уравнений вида $xy+...=0$), таким образом, не превосходит $16^3=4096$. В каждой такой системе приравняем соответствующие коэффициенты при первых степенях переменных. Получим набор систем уравнений относительно длин ребер. Отбракуем системы, дающие вырожденные решения, а также системы, влекущие метрическую симметричность хотя бы одного экватора. В тех случаях, когда уравнения распадаются в произведение уравнений, перейдем к рассмотрению соответствующих (более простых) систем. В результате придем к набору из 27 систем уравнений вида
\begin{equation}\label{3sys}
\left\{\begin{array}{l}
s \pm q \pm a \pm d=0\\
b \pm p \pm r \pm e=0\\
c \pm f \pm h \pm g=0,
\end{array}
\right.
\end{equation}
где в каждом уравнении независимо от других выбрана комбинация знаков, соответствующая одному из множителей в определении экватора с нулевой суммой (таких множителей три, отсюда и количество систем $3^3=27$). Таким образом, мы будем рассматривать {\it все} возможные случаи взаимного расположения экваторов с нулевой суммой. Поэтому тот факт, что на предыдущем шаге мы исключили из рассмотрения метрически симметричные экваторы, на самом деле не ограничивает общности рассуждения (любой метрически симметричный экватор является экватором с нулевой суммой).

Мы доказали, что утверждения 1, 2 и 3 доказываемой теоремы справедливы. Далее предположим, что не все экваторы метрически симметричны. Докажем в этом предположении справедливость утверждения 4. Как и раньше, рассмотрим только один из случаев (остальные разбираются аналогично).
Пусть, например,
\begin{equation}\label{3sys1}
\left\{\begin{array}{l}
s - q - a + d=0\\
b - p - r + e=0\\
c - f - h + g=0,
\end{array}
\right.
\end{equation}

Если при этом $b=r$ (и, следовательно, $p=e$), легко получить, учитывая утверждение 2 теоремы, что другие два экватора также должны быть метрически симметричными, что противоречит сделанному предположению. Поэтому можно считать, что $b\ne r$ (и $p\ne e$). Аналогично рассуждая, получаем, что и $c\ne h$ ($f\ne g$), и $s\ne a$ ($q\ne d$).

Рассмотрим три уравнения системы~(\ref{3sys1}) совместно с любыми тремя уравнениями утверждения 2. Решим полученную систему (в ней 6 уравнений и 12 неизвестных) относительно $d$, $q$, $f$, $g$, $p$, $e$ (это переменные, не вошедшие в выражение для $B_{\varepsilon_1, \varepsilon_2, \varepsilon_3}$):

\begin{equation}\label{subs}
\begin{array}{ccc}
d = {\displaystyle \frac{bh(a-s)}{bh-cr}} & f = {\displaystyle \frac{ra(c-h)}{ra-sb}} & p = {\displaystyle \frac{-cs(b-r)}{ah-cs}}\\
\mb & \mb & \mb\\
q = {\displaystyle \frac{cr(a-s)}{bh-cr}} & g = {\displaystyle \frac{sb(c-h)}{ra-sb}} & e = {\displaystyle \frac{-ah(b-r)}{ah-cs}}
\end{array}
\end{equation}
(Выражения в знаменателях не равны нулю: например, если бы $bh=cr$, то с учетом $q\ne d$ получили бы $bhq\ne cdr$.)

Вспомним, что в рассматриваемом случае (система~(\ref{3sys1})) мы имеем 6 пар условий на ребра, выражающие принадлежность трех экваторов типу III. Например, для $K_1$ это условия
$$c=q^2+p^2-\frac{pq}{sr}(h^2-r^2-s^2), \mr f=q^2+r^2- \frac{qr}{ps}(g^2-p^2-s^2).$$
С учетом~(\ref{subs}) последние равенства приобретают вид
$$\frac{c^2(bhs+ahr-crs-rsh)\cdot B_{1,1,1}}{(bh-cr)^2(ah-cs)^2}=0 \Leftrightarrow (bhs+ahr-crs-rsh)\cdot B_{1,1,1}=0$$
и
$$\frac{r^2(c-h)(bha+cba-cra-bcs)\cdot B_{1,1,1}}{(b-r)(ah-cs)(bh-cr)^2(ra-sb)}=0 \Leftrightarrow (bha+cba-cra-bcs)\cdot B_{1,1,1}=0$$

Если предположить, что $B_{1,1,1}\ne 0$, то необходимо
\begin{equation*}
\left\{\begin{array}{l}
bhs+ahr-crs-rsh=0\\
bha+cba-cra-bcs=0,
\end{array}
\right.
\end{equation*}
откуда, вычитая одно уравнение из другого получим невозможное равенство
$$(c+h)(b-r)(a-s)=0.$$

Таким образом, мы проверили, что в рассматриваемом случае $B_{1,1,1}=0$. Аналогично рассуждая, можно проверить, что и в остальных случаях один из множителей в утверждении 4 равен нулю. Утверждение 4 тем самым доказано, что и завершает доказательство всей теоремы.
\end{proof}

\begin{remark} \label{indp}
В теореме~\ref{t5} утверждение 4 не вытекает из утверждений 2 и 3: условия утверждений 2 и 3 задают многообразие с параметризацией~(\ref{subs}), где переменные $a,b,c,h,s,r$ свободны, а условие утверждения 4 задает связь между этими параметрами.
\end{remark}

\begin{remark}
Всякий набор значений параметров $a,b,c,h,s,r$, такой, что $B_{\varepsilon_1, \varepsilon_2, \varepsilon_3}=0$ (с некоторым выбором знаков $\varepsilon_k$) определяет, таким образом, метрику октаэдра указанного типа (конечно, все переменные должны уще удовлетворять некоторым условиям типа неравенства).
\end{remark}

\begin{remark}
В качестве $B_{\varepsilon_1, \varepsilon_2, \varepsilon_3}$ можно взять и другой многочлен: аналогичный по структуре, но от других переменных (подойдут ребра любой пары противоположных граней). Формулировка теоремы~\ref{t5} при этом сохраняется дословно, в доказательстве возникнет ровно та же система из шести уравнений, но решать ее нужно будет уже относительно других переменных (тех, что не вошли в выражение $B_{\varepsilon_1, \varepsilon_2, \varepsilon_3}$).
\end{remark}

\section{Классификация изгибаемых октаэдров и новые задачи}
\label{sec4}
\begin{definition}
Будем говорить, что октаэдр имеет:

тип А, если все его экваторы метрически симметричны;

тип Б, если все его экваторы имеют нулевую сумму, и справедливы равенства: $qge=pfd, \mr bhq=cdr, \mr rga=sfb$, $B_{1,1,1}\cdot B_{1,-1,-1}\cdot B_{-1,1,-1}\cdot B_{-1,-1,1}=0.$
\end{definition}

Из теорем~\ref{t2}---\ref{t5} непосредственно следует:
\begin{theorem}\label{t6}
Любой нетривиально изгибаемый октаэдр имеет тип A или Б.
\end{theorem}

Сравним полученный результат с результатом классической работы Брикара~\cite{bric}. Брикар доказал, что любой нетривиально изгибаемый октаэдр относится к одному из трех типов. Октаэдры Брикара первых двух типов имеют по шесть пар равных ребер и в нашей классификации относятся к типу А. Третий тип Брикара соответствует нашему типу Б.

Класс октаэдров типа А более широк по сравнению с классом октаэдров Брикара типов 1 и 2. На самом деле все наши формулировки и доказательства можно было уточнить, указывая каждый раз, какие конкретно пары равных ребер имеют экваторы. Но был сделан сознательный выбор в пользу лаконичности классификации и изложения, так как повторение классификации Брикара именно в части "простых"\ случаев (типы 1 и 2) не являлось приоритетной задачей.

Будет вполне ожидаемо, если среди комплексов типа A найдутся те, что не реализуются как изгибаемые октаэдры. Дело, однако, в том, что априори даже среди комплексов брикаровских типов могут найтись не реализуемые в виде изгибамых октаэров. Ведь в статье Брикара (как и в нашей) найдены лишь {\it необходимые} условия изгибаемости. В дополнение к этому Брикар проводит несколько конкретных геометрических построений, показывающих, что найденные им классы непусты. Но о достаточности найденных условий речь не идет. Геометрические построения Брикара годятся и для наших типов A и Б, и в этой части мы воспользуемся результатами Брикара, сказав, что классы изгибаемых октаэдров типов A и Б непусты.

Вместе с тем найденные необходимые условия являются содержательными. В работе~\cite{as1} установлено, что многочлен Сабитова для объема произвольного октаэдра имеет вид
$$Q(V)=V^8+a_7(l)V^7+...+a_1(l)V+a_0(l),$$
где $V$ --- квадрат объема, а коэффициенты $a_k(l)$ зависят от квадратов длин ребер. В общем случае (когда все ребра обозначены разными буквами) многочлен содержит несколько сотен миллионов слагаемых; есть алгоритм вычисления коэффициентов $a_k(l)$ для конкретных ребер. Так, в случае октаэдров типа A получаются (см.~\cite{gms})  многочлены, у которых $a_7(l)\ne 0$, и, возможно, еще $a_6(l)\ne 0$, а коэффициенты при младших степенях $V$ равны нулю. Случай октаэдров типа Б в~\cite{gms} не рассматривается. С использованием же формул вида~(\ref{subs}) удается найти, что в этом случае
$$a_7(l)=a_6(l)=0.$$
Остальные коэффициенты многочлена пока не найдены в виду значительных вычислительных сложностей. Отметим, что исходя из известных ранее описаний изгибаемых октаэдров (включая описание самого Брикара) получить и этот результат было бы трудно.

Сформулируем теперь две нерешенные задачи, при решении которых скорее всего пригодится найденное новое метрическое описание изгибаемых октаэдров.

1. Уже упомянутая задача о нахождении достаточных условий существования изгибаемой реализации данного метрического комплекса. Проиллюстрируем нетривиальность задачи парой фактов: (а) даже для комплекса, соответствующего октаэдру Брикару 1-го типа, во многих случаях существует как изгибаемая его реализация, так и неизгибаемая реализация (в виде выпуклого многогранника); (б) при этом существуют изгибаемые октаэдры Брикара 1-го типа, не имеющие выпуклых реализаций. Некоторые достаточные условия реализуемости получены в~\cite{mih1}. Но до полного исследования этого вопроса еще далеко.

2. Было бы интересно найти {\it все} коэффициенты многочлена для объема октаэдра типа Б. Согласно гипотезе, высказанной И.Х.~Сабитовым, для любого октаэдра типа Б {\it все} коэффициенты $a_k(l)$ равны нулю, и, соответственно уравнение для объема должно принять вид $V^8=0$. Помимо того, что этот факт представляет и самостоятельный интерес, отсюда сразу следовало бы, в частности, что соответствующая развертка не может быть реализована в виде выпуклого октаэдра.

Работа выполнена с использованием оборудования Центра коллективного пользования сверхвысокопроизводительными вычислительными ресурсами МГУ имени М.В. Ломоносова.

Автор статьи выражает искреннюю благодарность И.Х.~Сабитову и Д.И.~Сабитову за ценные замечания и советы.
\end{fulltext}

\end{document}